\documentclass[11pt,a4paper,reqno]{amsart}
\usepackage{amsfonts,caption,amsfonts,amsmath,amssymb,enumerate,hyperref,latexsym,amsthm,multicol,cases,xcolor}

 \newcommand{\Aut}{\mathrm{Aut}}
\newcommand{\bbF}{\mathbb{F}}
\newcommand{\Cay}{\mathrm{Cay}} \newcommand{\Cos}{\mathrm{Cos}}
\newcommand{\GL}{\mathrm{GL}}
\newcommand{\PG}{\mathrm{PG}} \newcommand{\PGL}{\mathrm{PGL}} \newcommand{\PSL}{\mathrm{PSL}}
\newcommand{\Val}{\mathrm{Val}}
\newcommand{\SL}{\mathrm{SL}}  \newcommand{\Sym}{\mathrm{Sym}}

\newtheorem{theorem}{Theorem}[section]

\newtheorem{lemma}[theorem]{Lemma}

\theoremstyle{definition}

\newtheorem{question}[theorem]{Question}
\newtheorem{conjecture}[theorem]{Conjecture}
\newtheorem{construction}[theorem]{Construction}
\newtheorem{problem}[theorem]{Problem}
\newtheorem*{remark}{Remark}

\begin{document}

\title[Non-diagonalizable digraphs]{A solution to Babai's problems on digraphs with non-diagonalizable adjacency matrix}

\thanks{*Corresponding author}

\author[Li]{Yuxuan Li}
\address{School of Mathematics and Statistics\\The University of Melbourne\\Parkville, VIC 3010\\Australia}
\email{yuxuan11@student.unimelb.edu.au}

\author[Xia]{Binzhou Xia}
\address{School of Mathematics and Statistics\\The University of Melbourne\\Parkville, VIC 3010\\Australia}
\email{binzhoux@unimelb.edu.au}

\author[Zhou]{Sanming Zhou}
\address{School of Mathematics and Statistics\\The University of Melbourne\\Parkville, VIC 3010\\Australia}
\email{sanming@unimelb.edu.au}

\author[Zhu]{Wenying Zhu*}
\address{Laboratory of Mathematics and Complex Systems (Ministry of Education)\\School of Mathematical Sciences\\Beijing Normal University\\
Beijing, 100875\\China}
\email{zfwenying@mail.bnu.edu.cn}

\begin{abstract}
The fact that the adjacency matrix of every finite graph is diagonalizable plays a fundamental role in spectral graph theory. Since this fact does not hold in general for digraphs, it is natural to ask whether it holds for digraphs with certain level of symmetry. Interest in this question dates back to the early 1980s, when P.~J.~Cameron asked for the existence of arc-transitive digraphs with non-diagonalizable adjacency matrix. This was answered in the affirmative by L.~Babai in 1985. Then Babai posed the open problems of constructing a 2-arc-transitive digraph and a vertex-primitive digraph whose adjacency matrices are not diagonalizable. In this paper, we solve Babai's problems by constructing an infinite family of $s$-arc-transitive digraphs for each integer $s\geq2$, and an infinite family of vertex-primitive digraphs, both of whose adjacency matrices are non-diagonalizable.


\end{abstract}
\maketitle
\section{Introduction}\label{sec1}

In this paper, a digraph $\Gamma$ is a pair $(V(\Gamma),\to)$ with $V(\Gamma)$ a set of vertices and $\to$ an irreflexive binary relation on $V(\Gamma)$, and all digraphs are assumed to be finite.
Suppose that $\Gamma$ has $n$ vertices $v_1,v_2,\ldots,v_n$.
The \emph{adjacency matrix} of $\Gamma$, denoted by $A(\Gamma)$, is the square matrix of order $n$ whose $(i,j)$-entry is $1$ if $v_i\to v_j$ and $0$ otherwise.
Note that the adjacency matrices of $\Gamma$ under different labellings of its vertex set are similar and hence have the same eigenvalues with multiplicities.
The eigenvalues of $A(\Gamma)$ are called the \emph{eigenvalues} of $\Gamma$.
The digraph $\Gamma$ is said to be \emph{diagonalizable} if its adjacency matrix is diagonalizable.

We say that $\Gamma$ is an \emph{undirected digraph} or a \emph{graph} if the binary relation $\to$ is symmetric.
For a graph, its adjacency matrix is symmetric, which makes it always diagonalizable.
Due to this essential property, the famous Courant-Fischer-Weyl Min-Max Theorem and Cauchy Interlacing Theorem, as powerful tools, are used frequently to deal with eigenvalues of graphs; refer to \cite{CRS2010} and~ \cite{GR2001}.
Compared with those of graphs, results about eigenvalues of digraphs are sparse due to the obvious fact that their adjacency matrices are not necessarily diagonalizable. It is natural to ask whether digraphs with certain prescribed properties are diagonalizable.
For example, some digraph properties in terms of association schemes guarantee that the digraph is diagonalizable; see~\cite{Suzuki2004,Wang2004} for instance.

For a non-negative integer $s$, an \emph{$s$-arc} of $\Gamma$ is a sequence $v_0,v_1,\dots,v_s$ of $s+1$ vertices with $v_i\rightarrow v_{i+1}$ for each $i\in\{0,1,\dots,s-1\}$. In particular, a \emph{$0$-arc} is a vertex of $\Gamma$. We say that $\Gamma$ is \emph{$s$-arc-transitive} if the automorphism group $\Aut(\Gamma)$ of $\Gamma$ acts transitively on the set of $s$-arcs of $\Gamma$.
The $0$-arc-transitive and $1$-arc-transitive digraphs are simply said to be \emph{vertex-transitive} and \emph{arc-transitive}, respectively.
For a finite group $G$ and a nonempty subset $S$ of $G\setminus\{1\}$, the \emph{Cayley digraph} on $G$ with \emph{connection set} $S$, denoted by $\Cay(G,S)$, is defined to be the digraph with vertex set $G$ such that $x\to y$ if and only if $yx^{-1}\in S$.

It is clear that every Cayley digraph is vertex-transitive as its automorphism group has a regular subgroup.
The first result exploring the relationship between symmetry and diagonalizability of digraphs was given by Godsil~\cite{Godsil1982} in 1982.
He proved that for each digraph $\Sigma$ with maximum degree greater than one, there exists a Cayley digraph $\Gamma$ such that the minimal polynomial of $A(\Sigma)$ divides that of $A(\Gamma)$. This implies the existence of non-diagonalizable Cayley digraphs, and thus non-diagonalizable vertex-transitive digraphs. On the other hand, a sufficient condition for a Cayley digraph to be diagonalizable is given by Babai in~\cite{Babai1995}, that is, if the connection set $S$ is closed under conjugation then $\Cay(G,S)$ is diagonalizable.

The digraph $\Gamma$ is said to be \emph{regular} if there exists a positive integer $d$, called the \emph{valency} of $\Gamma$ and denoted by $\Val(\Gamma)$, such that every vertex of $\Gamma$ has $d$ out-neighbours and $d$ in-neighbours.
Note that a regular $(s+1)$-arc-transitive digraph is also $s$-arc-transitive. In particular, a regular arc-transitive digraph is necessarily vertex-transitive.
In 1983, Cameron~\cite{Cameron1983} asked about the existence of non-diagonalizable arc-transitive digraphs. This was answered in the affirmative by Babai~\cite{Babai1985} in 1985.
In fact, Babai~\cite{Babai1985} proved a stronger result that for each integral matrix $A$, there exists an arc-transitive digraph $\Gamma$ such that the minimal polynomial of $A$ divides that of $A(\Gamma)$. In the same paper, he further posed the following open problems.
Recall that a permutation group $G$ on a set $\Omega$ is said to be primitive if $G$ does not preserve any nontrivial and proper partition of $\Omega$. We say that $\Gamma$ is \emph{vertex-primitive} if $\mathrm{Aut}(\Gamma)$ acts primitively on $V(\Gamma)$.

\begin{problem}\cite[Problem 1.4]{Babai1985}\label{pro1}
Construct a non-diagonalizable $2$-arc-transitive digraph and a non-diagonalizable vertex-primitive digraph.
\end{problem}

We remark that, for every positive integer $s$, the existence of non-diagonalizable $s$-arc-transitive digraphs can be deduced from the combination of some known results.
In 1975, Hoffman~\cite{Hoffman1975} showed that for each integral matrix $A$, there exists a digraph $\Gamma$ such that the minimal polynomial of $A$ divides that of $A(\Gamma)$.
For each digraph $\Gamma$, Godsil~\cite{Godsil1982} proved the existence of a regular digraph $\Sigma$ with the property that the minimal polynomial of $A(\Gamma)$ divides that of $A(\Sigma)$.
Moreover, for every regular digraph $\Sigma$, a result from Mansilla and Serra~\cite{MS2001} in 2001 shows that there exists an $s$-arc-transitive covering digraph $\Sigma_s$ of $\Sigma$ for each positive integer $s$.
Since the minimal polynomial of a digraph divides those of its covering digraphs (see~\cite[Corollary~3.3]{Babai1985}), we derive the existence of an $s$-arc-transitive digraph $\Sigma_s$ for each integral matrix $A$ and positive integer $s$ such that the minimal polynomial of $A$ divides that of $\Sigma_s$. This proves the existence of non-diagonalizable $s$-arc-transitive digraphs for each $s\ge1$. However, such a proof is not constructive.

In this paper, we solve Problem~\ref{pro1} by constructing infinite families of digraphs with the required properties.
To build an infinite family of digraphs from an existing one, we use the \emph{tensor product} $\Gamma\times\Sigma$ of digraphs $\Gamma$ and $\Sigma$, where $\Gamma\times\Sigma$  is the digraph with vertex set $V(\Gamma)\times V(\Sigma)$ such that $(u_1,v_1)\rightarrow(u_2,v_2)$ if and only if $u_1\rightarrow u_2$ in $\Gamma$ and $v_1\rightarrow v_2$ in $\Sigma$. For an integer $n\geq1$, denote by $\Gamma^{\times n}$ the tensor product of $n$ copies of digraph $\Gamma$.
Our main result gives infinite families of non-diagonalizable $s$-arc-transitive digraphs and non-diagonalizable vertex-primitive digraphs. The basic digraphs in these two families are as follows.

\begin{construction}\label{CONs-arc}
For each integer $s\geq2$, let $a_s=(2s-1,2s)(4s-1,4s)\in\Sym(4s)$, let $b_s=(1,3,5,\ldots,4s-1,2,4,6,\ldots,4s)\in\Sym(4s)$, let $R_s=\langle a_s,b_s\rangle$ be the group generated by $a_s$ and $b_s$, and let $\Gamma_s=\Cay(R_s,\{a_sb_s,b_s\})$.
\end{construction}

\begin{construction}\label{CONver-pri}
Let $R=\langle a,b\mid a^7=b^3=1,\,b^{-1}ab=a^2\rangle\times\langle c,d\mid c^7=d^3=1,\,d^{-1}cd=c^2\rangle$, let $\gamma$ be the automorphism of $R$ interchanging $a$ with $c$ and $b$ with $d$, let
\[
S=(S_1\cup S_1^{-1})(S_3\cup S_3^{-1})^\gamma\cup(S_3\cup S_3^{-1})(S_1\cup S_1^{-1})^\gamma\cup S_1S_2^\gamma\cup S_2S_1^\gamma\cup S_1^{-1}S_4^\gamma\cup S_4(S_1^{-1})^\gamma,
\]
where
\begin{align*}
S_1&=\{a,\,a^5,\,a^6b,\,a^6b^2\},\ \ S_2=\{ab,\,(ab)^{-1}\},\\
S_3&=\{a^3,\,b,\,ab^2,\,a^4b^2\},\ \ S_4=\{a^2b,\,(a^2b)^{-1}\},
\end{align*}
and let $\Sigma=\Cay(R,S)$.
\end{construction}

\begin{remark}
We will see in Lemmas~\ref{LEMcayR} and \ref{LEMRS} that the group $R_s$ in Construction~\ref{CONs-arc} is an extension of the elementary abelian group $C_2^s$ by the cyclic group $C_{2s}$, while $R$ and $S$ in Construction~\ref{CONver-pri} satisfy $R\cong (C_7\rtimes C_3)^2$ and $|S|=160$.
\end{remark}

Our main result is as follows.

\begin{theorem}\label{THEinfinite}
For all positive integers $n$ and $s\geq2$, with $\Gamma_s$ and $\Sigma$ defined in Constructions~$\ref{CONs-arc}$ and~$\ref{CONver-pri}$, the digraphs $\Gamma_s^{\times n}$ and $\Sigma^{\times n}$ satisfy the following:
\begin{enumerate}[\rm(a)]
\item $\Gamma_s^{\times n}$ is $s$-arc-transitive;
\item $\Sigma^{\times n}$ is vertex-primitive;
\item $\Gamma_s^{\times n}$ and $\Sigma^{\times n}$ are non-diagonalizable.
\end{enumerate}
\end{theorem}

The remainder of this paper is structured as follows. In the next section, we will give some basic definitions and lemmas that will play an important role in the proofs of our main results. After these preparations, we will prove in Section~\ref{sec3} that the digraph $\Gamma_s$ defined in Construction~$\ref{CONs-arc}$ is a non-diagonalizable $s$-arc-transitive digraph for each integer $s\geq2$ (see Theorem~\ref{THEmain3}), and prove in Section~\ref{sec4} that the digraph $\Sigma$ in Construction~$\ref{CONver-pri}$ is a non-diagonalizable vertex-primitive digraph (see Theorem~\ref{THEmain4}). Finally, Theorem~\ref{THEinfinite} follows from Lemma~\ref{LEMTensorofdi} and Theorems~\ref{THEmain3} and~\ref{THEmain4} immediately, as summarized in Section~\ref{sec5} along with some open questions and a conjecture.

\section{Preliminaries}\label{sec2}

Throughout the paper, $\sqcup$ denotes the disjoint union. For a positive integer $n$, denote the cyclic group of order $n$ by $C_n$ and the dihedral group of order $2n$ by $D_{2n}$.
Let $G$ be a finite group. For elements $a$ and $b$ in $G$ denote $a^b=b^{-1}ab$. For a subgroup $H$ and a subset $D$ of $G\setminus H$ such that $D$ is a union of double cosets of $H$ in $G$, the \emph{coset digraph} $\Cos(G,H,D)$ is the digraph with vertex set $[G{:}H]$, the set of right cosets of $H$ in $G$, and $Hx\to Hy$ if and only if $yx^{-1}\in D$. Clearly, the right multiplication action of $G$ on $[G{:}H]$ induces a group of automorphisms of $\Cos(G,H,D)$, and $\Cos(G,H,D)$ is arc-transitive if $D$ is a single double coset of $H$ in $G$.

For matrices $A$ and $B$, denote by $A\otimes B$ their \emph{Kronecker product (tensor product)}, that is, the matrix obtained by replacing each entry $a_{i,j}$ of $A$ with the block $a_{i,j}B$.
Let $\mathbb{C}$ be the complex field, and denote by $M_{n\times m}(\mathbb{C})$ the set of $n\times m$ matrices with entries in $\mathbb{C}$.
Some basic properties of the Kronecker product are given in the following lemma, which follows from~\cite[Equation~4.2.7,~Equation~4.2.8,~Lemma~4.2.10,~ Corollary~4.2.11 and Corollary~4.3.10]{HJ1994}.
\begin{lemma}\label{LEMtenpro}
Let $A\in M_{m\times n}(\mathbb{C})$, $B\in M_{p\times q}(\mathbb{C})$, $C\in M_{n\times k}(\mathbb{C})$ and $D\in M_{q\times r}(\mathbb{C})$. The following hold:
\begin{enumerate}[\rm(a)]
\item $(A\otimes B)(C\otimes D)=(AC)\otimes (BD)$;
\item if both $A$ and $B$ are invertible, then $A\otimes B$ is invertible and $(A\otimes B)^{-1}=A^{-1}\otimes B^{-1}$;
\item if $m=p$ and $n=q$, then $(A+B)\otimes U=A\otimes U+B\otimes U$ and $U\otimes (A+B)=U\otimes A+U\otimes B$ for all $U\in M_{\ell\times t}(\mathbb{C})$;
\item if $m=n$ and $p=q$, then $A\otimes B$ and $B\otimes A$ are similar.
\end{enumerate}
\end{lemma}

The constructions of $\Gamma_s$ and $\Sigma$ in Sections~\ref{sec3} and~\ref{sec4} are via Cayley digraphs. The following two lemmas enable us to prove the non-diagonalizability of Cayley digraphs on a group $G$ by analyzing irreducible representations (over $\mathbb{C}$) of $G$.
For a subset $S\subseteq G$ and two representations $\rho$ and $\varsigma$ of a group $G$, denote $\rho(S)=\sum_{s\in S}\rho(s)$ and $\rho(S)\oplus\varsigma(S)=\begin{pmatrix}\rho(S) & 0\\0 & \varsigma(S)\end{pmatrix}$.

\begin{lemma}\label{LEMrepre1}\cite[Proposition~7.1]{MA2011}
Let $G$ be a finite group, let $S$ be a nonempty subset of $G\setminus\{1\}$, and let $\{\rho_1,\dots,\rho_k\}$ be a complete set of irreducible representations of $G$ over $\mathbb{C}$. Then $A(\Cay(G,S))$ is similar to
 \[
 d_1\rho_1(S)\oplus d_2\rho_2(S)\oplus\cdots\oplus d_k\rho_k(S),
 \]
 where $d_i$ is the dimension of $\rho_i$ and $d_i\rho_i(S):=\underbrace{\rho_i(S)\oplus\cdots\oplus\rho_i(S)}_{d_i}$ for $i\in\{1,\ldots,k\}$.
\end{lemma}

\begin{lemma}\label{LEMrepre2}
Let $G$ be a finite group and let $S$ be a nonempty subset of $G\setminus\{1\}$. The digraph $\Cay(G,S)$ is non-diagonalizable if and only if there exists a representation $\rho$ of $G$ over $\mathbb{C}$ such that $\rho(S)$ is non-diagonalizable.
\end{lemma}

\begin{proof}
For each representation $\rho$ of $G$ over $\mathbb{C}$, by Maschke's theorem~(see~\cite[Corollary~1.6]{FH1991}), there exist irreducible representations $\rho_1,\rho_2,\ldots,\rho_t$ of $G$ satisfying $\rho=\rho_1\oplus\rho_2\oplus\cdots\oplus\rho_t$. This implies that $\rho(S)$ is non-diagonalizable if and only if $\rho_i(S)$ is non-diagonalizable for some $i\in \{1,2,\ldots,t\}$. According to Lemma~\ref{LEMrepre1}, the latter holds if and only if $A(\Cay(G,S))$ is non-diagonalizable. Thus the lemma follows.
\end{proof}

Denote by $J(\alpha,s)$ the $s\times s$ Jordan block with eigenvalue $\alpha$.

\begin{lemma}\label{LEMjorb}
If $s>1$ or $t>1$, then $J(\alpha,s)\otimes J(\beta,t)$ is non-diagonalizable.
\end{lemma}

\begin{proof}
According to~\cite[Theorem~4.3.17]{HJ1994}, $J(\alpha,s)\otimes J(\beta,t)$ is similar to
\[
\begin{cases}
J(0,\min\{s,t\})^{\oplus(|s-t|+1)}\oplus\bigoplus\limits^{\min\{s,t\}-1}_{k=1}J(0,k)^{\oplus2} &\text{if }\alpha=0=\beta\\
J(0,s)^{\oplus t} &\text{if }\alpha=0\neq\beta\\
J(0,t)^{\oplus s} &\text{if }\alpha\neq0=\beta\\
\bigoplus\limits^{\min\{s,t\}}_{k=1}J(\alpha\beta,s+t+1-2k) &\text{if }\alpha\beta\neq0.
\end{cases}
\]
This shows that the Jordan canonical form of $J(\alpha,s)\otimes J(\beta,t)$ contains one of the Jordan blocks $J(0,s)$, $J(0,t)$ and $J(\alpha\beta,s+t-1)$.
Since $s>1$ or $t>1$, we have $s+t-1>1$. Hence $J(\alpha,s)\otimes J(\beta,t)$ is non-diagonalizable.
\end{proof}

\begin{lemma}\label{tsnon-dia}
If either $\Gamma$ or $\Sigma$ is non-diagonalizable, then $\Gamma\times\Sigma$ is non-diagonalizable.
\end{lemma}

\begin{proof}
Let $A$ and $B$ be the adjacency matrices of $\Gamma$ and $\Sigma$, respectively.
Then $A\otimes B$ is the adjacency matrix of $\Gamma\times\Sigma$.
Suppose that either $\Gamma$ or $\Sigma$ is non-diagonalizable, that is, either $A$ or $B$ is non-diagonalizable.
This implies that there exist Jordan blocks $J(\alpha,s)$ of $A$ and $J(\beta,t)$ of $B$ with $s>1$ or $t>1$.
By Lemma~\ref{LEMtenpro}, each Jordan block of $J(\alpha,s)\otimes J(\beta,t)$ is a Jordan block of $A\otimes B$.
Thus we conclude from Lemma~\ref{LEMjorb} that $A\otimes B$ is non-diagonalizable, which means that $\Gamma\otimes\Sigma$ is non-diagonalizable.
\end{proof}

For a digraph $\Gamma$, recall the digraph $\Gamma^{\times n}$ defined in the paragraph before Construction~\ref{CONs-arc}.
We give some properties for $\Gamma^{\times n}$ in the following lemma.

\begin{lemma}\label{LEMTensorofdi}
For positive integers $n$ and $s$, the digraph $\Gamma^{\times n}$ satisfies the following:
\begin{enumerate}[\rm (a)]
\item $\Gamma^{\times n}$ is $s$-arc-transitive if $\Gamma$ is $s$-arc-transitive;
\item $\Gamma^{\times n}$ is vertex-primitive if $\Gamma$ is vertex-primitive and $|V(\Gamma)|$ is not prime;
\item $\Gamma^{\times n}$ is non-diagonalizable if $\Gamma$ is non-diagonalizable.
\end{enumerate}
\end{lemma}

\begin{proof}
Parts~(a) and (c) are obtained directly from \cite[Lemma~2.7]{GX2018} and Lemma~\ref{tsnon-dia}, respectively.
For part~(b), the conditions that $\Aut(\Gamma)$ is primitive on $V(\Gamma)$ and that $|V(\Gamma)|$ is not prime imply that $\Aut(\Gamma)\wr\Sym(n)$ is primitive on $V(\Gamma)^n$ (see \cite[Proposition~3.2~and~the~paragraph~thereafter]{Cameron1981}), and so $\Aut(\Gamma^{\times n})\geq\Aut(\Gamma)\wr\Sym(n)$ is primitive on $V(\Gamma)^n=V(\Gamma^{\times n})$.
\end{proof}

For each prime power $q$, we label the $1$-dimensional subspaces of $\mathbb{F}_q^2$ by the ratio of the coordinates, that is $\langle(x,1)\rangle$ is labelled by $x$ and $\langle(1,0)\rangle$ is labelled by $\infty$. The set of $1$-spaces is then identified with the set $\mathbb{F}_q\cup\{\infty\}$, called the \emph{projective line} over $\mathbb{F}_q$, and denoted by $\PG(1,q)$.
For each matrix $A=\left(\begin{matrix} a & b\\ c & d\end{matrix}\right)\in\GL(2,q)$, the transformation
\[
\phi_A\colon\PG(1,q)\rightarrow\PG(1,q),\ x\mapsto\frac{ax+c}{bx+d}
\]
is called a \emph{linear fractional transformation} on $\PG(1,q)$.
Here we set $\phi_A(\infty)=a/b$ and $\phi_A\left(-d/b\right)=\infty$ if $b\neq0$, and set $\phi_A(\infty)=\infty$ if $b=0$. Note that
\[
\phi\colon\GL(2,q)\rightarrow \Sym(\PG(1,q)),\ A\mapsto\phi_A
\]
is a group homomorphism, and we have $\PGL(2,q)=\phi(\GL(2,q))$ and $\PSL(2,q)=\phi(\SL(2,q))$.
Moreover, $\phi_A\in\PSL(2,q)$ if and only if $\det(A)$ is a square in $\bbF_q$.

\section{The non-diagonalizable $s$-arc-transitive digraphs $\Gamma_s$}\label{sec3}

Fix an integer $s\geq2$. For simplicity of notation, let $a=a_s=(2s-1,2s)(4s-1,4s)\in\Sym(4s)$, let $b=b_s=(1,3,5,\ldots,4s-1,2,4,6,\ldots,4s)\in\Sym(4s)$, and let $R=R_s=\langle a,b\rangle$. Then $\Gamma_s=\Cay(R,\{ab,b\})$ is as defined in Construction~\ref{CONs-arc}.

Throughout this section, let $N=\langle a,a^{b},a^{b^2},\ldots,a^{b^{s-1}}\rangle$, let $G=\langle h,g\rangle$ with
\[
h=(1,2)\in\Sym(4s)\ \text{ and }\ g=(1,3,5,\ldots,4s-1)(2,4,6,\ldots,4s)\in\Sym(4s),
\]
and let $H=\langle h,h^{g},h^{g^2},\ldots,h^{g^{s-1}}\rangle$. Observe that
\begin{align}\label{Eq5}
H&=\langle h\rangle\times\langle h^g\rangle\times\langle h^{g^2}\rangle\times\cdots\times\langle h^{g^{s-1}}\rangle\nonumber\\
&=\langle (1,2)\rangle\times\langle (3,4)\rangle\times\langle (5,6)\rangle\times\cdots\times\langle (2s-1,2s)\rangle\cong C_2^s,
\end{align}
$a=h^{g^{s-1}}h^{g^{-1}}$, and $b=gh$. In particular, $R\leq G$.

\begin{lemma}\label{LEMcayR}
The subgroup $N=\langle a\rangle\times\langle a^{b}\rangle\times\langle a^{b^2}\rangle\times\cdots\times\langle a^{b^{s-1}}\rangle\cong C_2^s$ is normal in $R$ with $R/N=\langle bN\rangle\cong C_{2s}$. In particular, $|R|=2^{s+1}s$.
\end{lemma}

\begin{proof}
Note that $a$ has order $|a|=2$ and $a^{b^\ell}=(2\ell-1,2\ell)(2s+2\ell-1,2s+2\ell)$ for each $\ell\in \{1,2,\ldots,s\}$.
We see that
\[
N=\langle a\rangle\times\langle a^{b}\rangle\times\langle a^{b^2}\rangle\times\cdots\times\langle a^{b^{s-1}}\rangle\cong C_2^s
\]
is normalized by $a$ and $b$, and so $N\trianglelefteq\langle a,b\rangle=R$. This together with $a\in N$ and
\[
b^{2s}=(1,2)(3,4)\cdots(4s-1,4s)=aa^{b}a^{b^2}\cdots a^{b^{s-1}}\in N
\]
leads to $R/N=\langle bN\rangle\cong C_{2s}$. As a consequence, $|R|=2^s\cdot2s$.
\end{proof}

In order to prove that $\Gamma_s$ is $s$-arc-transitive, we need the following lemma.

\begin{lemma}\label{LEMCosdi}
For the digraph $\Gamma_s$ in Construction~$\ref{CONs-arc}$, we have $\Gamma_s\cong\Cos(G,H,HgH)$.
\end{lemma}

\begin{proof}
Since $ghg^{-1}=(4s-1,4s)\notin H$, we have $H\neq Hghg^{-1}$ and so $Hg\neq Hgh$. Hence $HgH\supseteq Hg\sqcup Hgh$. Moreover, since
\begin{align*}
H\cap H^g&=(\langle h\rangle\times\langle h^g\rangle\times\cdots\times\langle h^{g^{s-1}}\rangle)\cap
(\langle h^g\rangle\times\langle h^{g^2}\rangle\times\cdots\times\langle h^{g^s}\rangle)\\
&=\langle h^g\rangle\times\langle h^{g^2}\rangle\times\cdots\times\langle h^{g^{s-1}}\rangle\\
&\cong C_2^{s-1},
\end{align*}
we see that $|HgH|/|H|=|H|/|H\cap H^{g}|=2$. Therefore, $HgH=Hg\sqcup Hgh$. As $a=h^{g^{s-1}}h^{g^{-1}}$ and $b=gh$, we have $ab=h^{g^{s-1}}g\in Hg$, and so
\begin{equation}\label{Eq4}
HgH=Hg\sqcup Hgh=Hab\sqcup Hb.
\end{equation}

Now we prove that $|G|=2^{2s}\cdot2s$. Let $M=\langle h,h^{g},h^{g^2},\ldots,h^{g^{2s-1}}\rangle$.
Since $h^{g^i}=(2i+1,2i+2)$ for $i\in \{0,1,\ldots,2s-1\}$, we see that
\[
M=\langle h\rangle\times\langle h^g\rangle\times\langle h^{g^2}\rangle\times\cdots\times\langle h^{g^{2s-1}}\rangle\cong C_2^{2s}.
\]
Therefore, $M$ is normalized by both $h$ and $g$ as $g^{2s}=1$, and so $M\trianglelefteq\langle h,g\rangle=G$.
Together with the facts that $g$ has order $2s$, that elements in $M$ have order dividing $2$ and that
\[
g^s=(1,2s+1)(2,2s+2)\cdots(2s-1,4s-1)(2s,4s)\notin M,
\]
this gives $G/M=\langle gM\rangle\cong C_{2s}$, and so $|G|=2^{2s}\cdot2s$.

Let $\psi\colon r\mapsto Hr$ be the mapping from the vertex set $R$ of $\Gamma_s$ to $[G{:}H]$. Next we prove that $\psi$ is a graph isomorphism from $\Gamma_s$ to $\Cos(G,H,HgH)$.
Note from
\[
N=\langle (1,2)(2s+1,2s+2)\rangle\times\langle (3,4)(2s+3,2s+4)\rangle\times\cdots\times\langle (2s-1,2s)(4s-1,4s)\rangle
\]
and $H=\langle (1,2)\rangle\times\langle (3,4)\rangle\times\cdots\times\langle (2s-1,2s)\rangle$ that $N\cap H=1$.
Since $R\cap H\leq H$ is an elementary abelian $2$-group, the only possible non-identity elements of $R\cap H$ are involutions.
Moreover, we deduce from $R/N=\langle bN\rangle\cong C_{2s}$ that the involutions of $R$ are contained in $N\cup b^sN$.
Thus $R\cap H\subseteq (N\cup b^sN)\cap H$. Note from
\[
b^s=(1,2s+1,2,2s+2)(3,2s+3,4,2s+4)\cdots(2s-1,4s-1,2s,4s)
\]
that $b^sN\cap H=1$. Hence $R\cap H\subseteq (N\cup b^sN)\cap H=1$ as $N\cap H=1$. From Lemma~\ref{LEMcayR} we have $|R|=2^{s+1}s$.
Since $R\leq G$ and
\[
|G|=2^{2s}\cdot2s=2^{s+1}s\cdot2^s=|R||H|,
\]
we conclude that $R$ forms a right transversal of $H$ in $G$, and so the mapping $\psi$ is bijective.
Hence for $r_1$ and $r_2$ in $R$, we have $r_2r_1^{-1}\in\{ab,b\}$ if and only if $Hr_2r_1^{-1}\subseteq Hab\sqcup Hb$.
By~\eqref{Eq4}, the latter condition holds if and only if $Hr_2r_1^{-1}\subseteq HgH$, or equivalently, $r_2r_1^{-1}\in HgH$.
Thus we conclude that $r_1\rightarrow r_2$ is an arc in $\Gamma_s$ if and only if $Hr_1\rightarrow Hr_2$ is an arc of $\Cos(G,H,HgH)$.
This shows that $\psi$ is an isomorphism from $\Gamma_s$ to $\Cos(G,H,HgH)$.
\end{proof}

Now we give the main result of this section.

\begin{theorem}\label{THEmain3}
For the digraph $\Gamma_s$ in Construction~$\ref{CONs-arc}$, the following hold:
\begin{enumerate}[\rm(a)]
\item $|V(\Gamma_s)|=2^{s+1}s$;
\item $\Val(\Gamma_s)=2$;
\item $\Gamma_s$ is strongly connected;
\item $\Gamma_s$ is $s$-arc-transitive;
\item $\Gamma_s$ is non-diagonalizable.
\end{enumerate}
\end{theorem}

\begin{proof}
Parts~(a) and (b) follow directly from $\Gamma_s=\Cay(R,\{ab,b\})$ and $|R|=2^s\cdot2s$.
Since $\langle ab,b\rangle=\langle a,b\rangle=R$, we see that $\Gamma_s$ is a connected digraph, which implies that $\Gamma_s$ is strongly connected (see~\cite[Lemma~2.6.1]{GR2001}), as part~(c) states. Note that
\[
H\rightarrow Hg\rightarrow\cdots\rightarrow Hg^{s-1}\rightarrow Hg^s
\]
is an $s$-arc of the coset digraph $\Cos(G,H,HgH)$ and the stabilizer in $G$ of this $s$-arc is $H\cap H^g\cap\cdots\cap H^{g^s}$. It is clear from~\eqref{Eq5} that
\begin{align*}
H^g\cap H^{g^2}\cap\cdots\cap H^{g^i}&=\langle h^{g^i}\rangle\times\langle h^{g^{i+1}}\rangle\times\cdots\times\langle h^{g^s}\rangle,\\
H\cap H^g\cap\cdots\cap H^{g^i}&=\langle h^{g^i}\rangle\times\langle h^{g^{i+1}}\rangle\times\cdots\times\langle h^{g^{s-1}}\rangle,\\
H^g\cap H^{g^2}\cap\cdots\cap H^{g^{i+1}}&=\langle h^{g^{i+1}}\rangle\times\langle h^{g^{i+2}}\rangle\times\cdots\times\langle h^{g^s}\rangle,
\end{align*}
and so
\[
H^g\cap H^{g^2}\cap\cdots\cap H^{g^i}=(H\cap H^g\cap\cdots\cap H^{g^i})(H^g\cap H^{g^2}\cap\cdots\cap H^{g^{i+1}})
\]
for each $i\in\{0,1,\ldots,s-1\}$. Recall that $G\leq\Aut\big(\Cos(G,H,HgH)\big)$ and $\Cos(G,H,HgH)$ is arc-transitive.
Thus, by \cite[Lemma~2.2]{GX2018} and Lemma~\ref{LEMCosdi}, we conclude that $\Gamma_s$ is $s$-arc-transitive, as part~(d) asserts.

Now it remains to prove part~(e). Denote $a_k=a^{b^k}$ for $k\in \{1,2,\ldots,s\}$.
According to Lemma~\ref{LEMcayR}, any elements $x$ and $y$ in $R$ can be written as
\begin{equation}\label{Eq1}
x=a_1^{\varepsilon_1}a_2^{\varepsilon_2}\cdots a_s^{\varepsilon_s}b^m
\ \ \text{and}\ \
y=a_1^{\theta_1}a_2^{\theta_2}\cdots a_s^{\theta_s}b^n
\end{equation}
for some $\varepsilon_1,\varepsilon_2,\ldots,\varepsilon_s,\theta_1,\theta_2,\ldots,\theta_s \in \{0,1\}$ and $m,n\in \{1,2,\ldots,2s\}$.
Since $(a_k)^{b^\ell}=(a^{b^k})^{b^\ell}=a^{b^{k+\ell}}=a_{k+\ell}$, we have
\begin{align}\label{Eq3}
xy
&=a_1^{\varepsilon_1}a_2^{\varepsilon_2}\cdots a_s^{\varepsilon_s}(a_1^{\theta_1}a_2^{\theta_2}\cdots a_s^{\theta_s})^{b^{-m}}b^mb^{n}\\
&=a_1^{\varepsilon_1}a_2^{\varepsilon_2}\cdots a_s^{\varepsilon_s}a_{1-m}^{\theta_1}a_{2-m}^{\theta_2}\cdots a_{s-m}^{\theta_s}b^{m+n}\nonumber\\
&=a_1^{\varepsilon_1+\theta_{1+m}}a_2^{\varepsilon_2+\theta_{2+m}}\cdots a_s^{\varepsilon_s+\theta_{s+m}}b^{m+n},\nonumber
\end{align}
where subscripts are counted modulo $s$.

First assume that $s\geq3$ is odd. In this case, let $V$ be the vector space over $\mathbb{C}$ with basis $e_1,e_2,\ldots,e_s$,
and for $x$ as in~\eqref{Eq1}, let $\rho(x)$ be the linear transformation on $V$ such that
\begin{equation}\label{Eq2}
e_i^{\rho(x)}=(-1)^{-\varepsilon_{2-2i}+\sum^s_{k=1}\varepsilon_k}e_{i+m(s-1)/2}\ \text{ for all }i\in\{1,2,\dots,s\},
\end{equation}
where subscripts are counted modulo $s$. It follows that
\begin{align*}
(e_i^{\rho(x)})^{\rho(y)}
&=((-1)^{-\varepsilon_{2-2i}+\sum^s_{k=1}\varepsilon_k}e_{i+m(s-1)/2})^{\rho(y)}\\
&=(-1)^{-\varepsilon_{2-2i}+\sum^s_{k=1}\varepsilon_k}e_{i+m(s-1)/2}^{\rho(y)}\\
&=(-1)^{-\varepsilon_{2-2i}+\sum^s_{k=1}\varepsilon_k}(-1)^{-\theta_{2-2(i+m(s-1)/2)}+\sum^s_{k=1}\theta_k}e_{i+m(s-1)/2+n(s-1)/2}\\
&=(-1)^{-\varepsilon_{2-2i}+\sum^s_{k=1}\varepsilon_k}(-1)^{-\theta_{2-2i+m}+\sum^s_{k=1}\theta_k}e_{i+(m+n)(s-1)/2}\\
&=(-1)^{-(\varepsilon_{2-2i}+\theta_{2-2i+m})+\sum^s_{k=1}(\varepsilon_k+\theta_{k+m})}e_{i+(m+n)(s-1)/2}\\
&=e_i^{\rho(xy)}
\end{align*}
for $i\in \{1,2,\ldots,s\}$. Hence $\rho$ is a representation of $R$ on $V$.
For $i,j\in\{1,2,\dots,s\}$, let $E_{i,j}$ be the $s\times s$ matrix with $(i,j)$-entry $1$ and other entries $0$.
With respect to the basis $e_1,e_2,\ldots,e_s$ we deduce from~\eqref{Eq2} that
\[
\rho(ab)=\rho(a_sb)=
\begin{pmatrix}
&-I_{(s+1)/2}\\
-I_{(s-1)/2} &\\
\end{pmatrix}
+2E_{1,(s+1)/2}\ \text{ and }\ \rho(b)=\begin{pmatrix}
&I_{(s+1)/2}\\
I_{(s-1)/2} &\\
\end{pmatrix},
\]
which yields
\[
\rho(ab)+\rho(b)=2E_{1,(s+1)/2}.
\]
Since $2E_{1,(s+1)/2}$ is non-diagonalizable, we conclude from Lemma~\ref{LEMrepre2} that $\Gamma_s$ is non-diagonalizable.

Next assume that $s\geq2$ is even.
For each integer $t$, denote
\[
\overline{t}=t\bmod2=
\begin{cases}
0, \ &
\ \text{ if $t$ is even, }\\
1, \ &
\ \text{ if $t$ is odd. }
\end{cases}
\]
In this case, let $V$ be the vector space over $\mathbb{C}$ with basis $e_1,e_2$,
and for $x$ as in~\eqref{Eq1}, let $\rho(x)$ be the linear transformation on $V$ such that
\[
e_i^{\rho(x)}=(-1)^{\delta_i}(\overline{m+i})e_1+(-1)^{\delta_i}(\overline{m+i+1})e_2 \ \text{ for all }i\in\{1,2\},
\]
where $\delta_i=\sum^{s/2-1}_{k=0}\varepsilon_{2k+i}$ for $i\in \{1,2\}$.
Thus by~\eqref{Eq1} and~\eqref{Eq3}, with respect to the basis $e_1,e_2$ we have
\[
\rho(x)=
\begin{pmatrix}
(-1)^{\delta_1}\overline{m+1}&(-1)^{\delta_1}\overline{m}\\
(-1)^{\delta_2}\overline{m}&(-1)^{\delta_2}\overline{m+1}\\
\end{pmatrix},
\]
\[
\rho(y)=
\begin{pmatrix}
(-1)^{\sigma_1}\overline{n+1}&(-1)^{\sigma_1}\overline{n}\\
(-1)^{\sigma_2}\overline{n}&(-1)^{\sigma_2}\overline{n+1}\\
\end{pmatrix},
\]
\[
\rho(xy)=
\begin{pmatrix}
(-1)^{\gamma_1}\overline{m+n+1}&(-1)^{\gamma_1}\overline{m+n}\\
(-1)^{\gamma_2}\overline{m+n}&(-1)^{\gamma_2}\overline{m+n+1}\\
\end{pmatrix},
\]
where $\sigma_i=\sum^{s/2-1}_{k=0}\theta_{2k+i}$ and $\gamma_i=\sum^{s/2-1}_{k=0}\big(\varepsilon_{2k+i}+\theta_{2k+i+m}\big)$ with the subscripts of $\theta$ counted modulo s for $i\in \{1,2\}$.
Since $\gamma_i=\delta_i+(\overline{m+i})\sigma_1+(\overline{m+i+1})\sigma_2$, a straightforward calculation shows that $\rho(xy)=\rho(x)\rho(y)$.
Hence $\rho$ is a representation of $R$ on $V$. Since
\[
\rho(ab)+\rho(b)=
\begin{pmatrix}
0&1\\
-1&0\\
\end{pmatrix}
+
\begin{pmatrix}
0&1\\
1&0\\
\end{pmatrix}
=
\begin{pmatrix}
0&2\\
0&0\\
\end{pmatrix}
\]
is non-diagonalizable, we conclude from Lemma~\ref{LEMrepre2} that $\Gamma_s$ is non-diagonalizable. This completes the proof of part~(e).
\end{proof}

\section{The non-diagonalizable vertex-primitive digraph $\Sigma$}\label{sec4}

Recall from Construction~\ref{CONver-pri} that $\Sigma=\Cay(R,S)$ with
\begin{align*}
R&=\langle a,b\mid a^7=b^3=1,\,b^{-1}ab=a^2\rangle\times\langle c,d\mid c^7=d^3=1,\,d^{-1}cd=c^2\rangle,\\
S&=(S_1\cup S_1^{-1})(S_3\cup S_3^{-1})^\gamma\cup(S_3\cup S_3^{-1})(S_1\cup S_1^{-1})^\gamma\cup S_1S_2^\gamma\cup S_2S_1^\gamma\cup S_1^{-1}S_4^\gamma\cup S_4(S_1^{-1})^\gamma,
\end{align*}
where $\gamma$ is the automorphism of $R$ interchanging $a$ with $c$ and $b$ with $d$, and
\begin{align*}
S_1&=\{a,\,a^5,\,a^6b,\,a^6b^2\},\ \ S_2=\{ab,\,(ab)^{-1}\},\\
S_3&=\{a^3,\,b,\,ab^2,\,a^4b^2\},\ \ S_4=\{a^2b,\,(a^2b)^{-1}\}.
\end{align*}

The following lemma gives basic properties on $R$ and $S$.

\begin{lemma}\label{LEMRS}
The following hold:
\begin{enumerate}[\rm(a)]
\item $R\cong(C_7\rtimes C_3)^2$;
\item $S=(S_1\sqcup S_1^{-1})(S_3\sqcup S_3^{-1})^\gamma\sqcup(S_3\sqcup S_3^{-1})(S_1\sqcup S_1^{-1})^\gamma\sqcup S_1S_2^\gamma\sqcup S_2S_1^\gamma\sqcup S_1^{-1}S_4^\gamma\sqcup S_4(S_1^{-1})^\gamma$;
\item $|S|=160$.
\end{enumerate}
\end{lemma}

\begin{proof}
Part~(a) is obvious. Next we prove parts~(b) and~(c).
Observe that $S_i\cap S_j=\emptyset$ for all $i,j\in \{1,2,3,4\}$ with $i\neq j$. Moreover, since
\[
S_1^{-1}=\{a^6,\,a^2,\,a^2b^2,\,a^4b\}\ \text{ and }\ S_3^{-1}=\{a^4,\,b^2,\,a^3b,\,a^5b\},
\]
we observe that the sets $S_1\cup S_2\cup S_3\cup S_4$, $S_1^{-1}$ and $S_3^{-1}$ are pairwise disjoint.
Thus
\[
S=(S_1\sqcup S_1^{-1})(S_3\sqcup S_3^{-1})^\gamma\sqcup(S_3\sqcup S_3^{-1})(S_1\sqcup S_1^{-1})^\gamma\sqcup S_1S_2^\gamma\sqcup S_2S_1^\gamma\sqcup S_1^{-1}S_4^\gamma\sqcup S_4(S_1^{-1})^\gamma,
\]
proving part~(b). As a consequence,
\begin{align*}
|S|&=|(S_1\sqcup S_1^{-1})||(S_3\sqcup S_3^{-1})^{\gamma}|+|(S_3\sqcup S_3^{-1})||(S_1\sqcup S_1^{-1})^\gamma|\\
&\phantom{=}+|S_1||S_2^\gamma|+|S_2||S_1^\gamma|+|S_1^{-1}||S_4^\gamma|+|S_4||S_1^{-\gamma}|\\
&=(4+4)\cdot(4+4)+(4+4)\cdot(4+4)+4\cdot2+4\cdot2+4\cdot2+4\cdot2\\
&=160,
\end{align*}
as part~(c) states.
\end{proof}

Recall the group homomorphism $\phi\colon\GL(2,q)\rightarrow\PGL(2,q)\leq\Sym(\PG(1,q))$ defined at the end of Section~\ref{sec2}. For our convenience, we identify $R$ with a permutation group on $\PG(1,7)\times\PG(1,7)$ by letting
\[
a\colon(x,y)\mapsto(x+1,\,y),\ b\colon(x,y)\mapsto(2x,\,y),\ c\colon(x,y)\mapsto(x,\,y+1),\ d\colon(x,y)\mapsto(x,\,2y).
\]
We also fix the following notation throughout this section. Let
\[
s\colon(x,y)\mapsto\left(\frac{2x+1}{x+1},\,y\right),\ t\colon(x,y)\mapsto\left(\frac{-1}{x},\,y\right),\ \alpha\colon (x,y)\mapsto\left(\frac{-x}{x+1},\,\frac{-y}{y+1}\right)
\]
be elements of $\PGL(2,7)\times\PGL(2,7)$, and let
\[
\beta\colon(x,y)\mapsto(y,x)
\]
be a permutation on $\PG(1,7)\times\PG(1,7)$. Then $R$ is normalized by $\beta$,
and the automorphism of $R$ induced by $\beta$ is equal to $\gamma$ (recall that $\gamma$ is the automorphism of $R$ interchanging $a$ with $c$ and $b$ with $d$). Let
\begin{equation}\label{Eq12}
u=s^\beta,\ v=t^\beta,\ g_1=a^4c^5,\ g_2=a^2c^3d^2,
\end{equation}
and let
\[
G=\langle a,b,c,d,t,v,\alpha,\beta\rangle,\ H=\langle s,t,u,v,\alpha,\beta\rangle.
\]

Under the above notation, we have the following lemma.

\begin{lemma}\label{LEMorder}
We have $|s^2|=|u^2|=|t|=|v|=|\alpha|=|\beta|=2$, $s=(\alpha t)^2$ and $u=(\alpha v)^2$.
\end{lemma}

\begin{proof}
From the definitions of $\beta$, $u$ and $v$ we see that $|\beta|=2$, $|u|=|s|$ and $|v|=|t|$.
Note that
\[
s=(\phi\times\phi)\left(\left(\begin{matrix} 2&1\\1&1\end{matrix}\right), I\right),\ t=(\phi\times\phi)\left(\left(\begin{matrix} 0& 1 \\ -1& 0\end{matrix}\right), I\right),\ \alpha=(\phi\times\phi)\left(\left(\begin{matrix} -1&1\\0& 1\end{matrix}\right),\left(\begin{matrix} -1&1\\0& 1\end{matrix}\right)\right).
\]
It follows from
\[
\left(\left(\begin{matrix} 2&1\\1&1\end{matrix}\right)^2\right)^2=\left(\begin{matrix} 5&3\\3&2\end{matrix}\right)^2=
\left(\begin{matrix} -1&0\\0&-1\end{matrix}\right)=\left(\begin{matrix} 0&1\\-1&0\end{matrix}\right)^2\ \text{ and } \
\left(\begin{matrix} -1&1\\0& 1\end{matrix}\right)^2=\left(\begin{matrix} 1&0\\0& 1\end{matrix}\right)
\]
that $|s^2|=|t|=|\alpha|=2$. As a consequence, $|u^2|=|v|=2$. Moreover,
\begin{align*}
(\alpha t)^2&=(\phi\times\phi)\left(\left(\left(\begin{matrix} -1&1\\0&1\end{matrix}\right)\left(\begin{matrix} 0&1\\-1&0\end{matrix}\right)\right)^2, \left(\begin{matrix} -1&1\\0&1\end{matrix}\right)^2\right)\\
&=(\phi\times\phi)\left(\left(\begin{matrix} -1&-1\\-1&0\end{matrix}\right)^2, I\right)\\
&=(\phi\times\phi)\left(\left(\begin{matrix} 2&1\\1&1\end{matrix}\right), I\right)=s.
\end{align*}
This together with the observation $\alpha^\beta=\alpha$ yields $u=s^\beta=((\alpha t)^2)^\beta=(\alpha^\beta t^\beta)^2=(\alpha v)^2$.
\end{proof}

From the definition of $G$ and the previous lemma, we see that $H$ and $R$ are both subgroups of $G$. The following lemma reveals the relation between $G$, $H$ and $R$.

\begin{lemma}\label{LEMmaxsub}
The group $H=(\langle s,t\rangle\times\langle u,v\rangle)\rtimes(\langle\alpha\rangle\times\langle\beta\rangle)\cong(D_8\times D_8)\rtimes C_2^2$ is maximal in $G\cong(\PSL(2,7)\times\PSL(2,7))\rtimes C_2^2$ with right transversal $R$.
\end{lemma}

\begin{proof}
It is straightforward to verify that $s^t=s^{-1}$, $u^v=u^{-1}$, and
\[
H=(\langle s,t\rangle\times\langle u,v\rangle)\rtimes(\langle\alpha\rangle\times\langle\beta\rangle).
\]
Since $s^t=s^{-1}$ and $u^v=u^{-1}$, we derive from Lemma~\ref{LEMorder} that $\langle s,t\rangle\cong\langle u,v\rangle\cong D_8$, and so
\[
H\cong(D_8\times D_8)\rtimes C_2^2.
\]
In particular, $|H|=2^8$. From Lemma~\ref{LEMorder} we see that $s=(\alpha t)^2$ and $u=(\alpha v)^2$. Hence $H\leq G$.
Observe that
\[
a=(\phi\times\phi)\left(\left(\begin{matrix} 1&0\\1&1\end{matrix}\right), I\right),\ b=(\phi\times\phi)\left(\left(\begin{matrix} 2& 0 \\ 0& 1\end{matrix}\right), I\right),\ t=(\phi\times\phi)\left(\left(\begin{matrix} 0&1\\-1& 0\end{matrix}\right),I\right).
\]
Since $\langle a,b\rangle\cong C_7\rtimes C_3$ has index $8$ in $\PSL(2,7)$ and its order is coprime to $|t|=2$, it follows that the index of $\langle a,b,t\rangle$ in $\PSL(2,7)$ is at most $4$. Since $\PSL(2,7)$ is a simple group of order $168$, we then obtain $\langle a,b,t\rangle\cong\PSL(2,7)$.
Moreover,
\[
\alpha=(\phi\times\phi)\left(\left(\begin{matrix} -1&1\\0& 1\end{matrix}\right), \left(\begin{matrix} -1&1\\0& 1\end{matrix}\right)\right)\ \ \text{with}\ \ \det\left(\begin{matrix} -1&1\\0&1\end{matrix}\right)=-1,
\]
and $-1$ is not a square in $\bbF_7$. We conclude that $\langle a,b,t,\alpha\rangle\cong\langle c,d,v,\alpha\rangle\cong\PGL(2,7)$ and
\[
G=(\langle a,b,t\rangle\times\langle c,d,v\rangle)\rtimes(\langle\alpha\rangle\times\langle\beta\rangle),
\]
which implies that
\[
G\cong(\PSL(2,7)\times\PSL(2,7))\rtimes C_2^2.
\]
In particular, $|G|=2^8\cdot3^2\cdot7^2$.

By Lemma~\ref{LEMorder}, we have $s=(\alpha t)^2$, $|t|=|\alpha|=2$ and $|s|=4$, it follows that $\langle s,t,\alpha\rangle=\langle t,\alpha\rangle\cong D_{16}$.
Let $M$ be a maximal subgroup of $\langle a,b,t,\alpha\rangle\cong\PGL(2,7)$ containing $\langle s,t,\alpha\rangle$.
Then $|\PGL(2,7)|/|M|$ divides $|\PGL(2,7)|/|\langle s,t,\alpha\rangle|=336/16=21$.
If $|\PGL(2,7)|/|M|=3$, then $M$ would contain $\PSL(2,7)$, not possible. Moreover, \cite[Table~2.1]{DM1996} shows that $\langle a,b,t,\alpha\rangle\cong\PGL(2,7)$ has no subgroup of index $7$. Thus $|\PGL(2,7)|/|M|=21$, and so $D_{16}\cong\langle s,t,\alpha\rangle=M$. Since $M$ is maximal in $\langle a,b,t,\alpha\rangle$, it follows that $\langle u,v,\alpha\rangle=\langle s,t,\alpha\rangle^\beta$ is maximal in $\langle a,b,t,\alpha\rangle^\beta=\langle c,d,v,\alpha\rangle$. As a consequence, $H$ is maximal in $G$.

Finally, the fact that $|R|=3^2\cdot7^2$ is prime to $|H|=2^8$ yields $R\cap H=1$. This together with $|G|=2^8\cdot3^2\cdot7^2=|H||R|$ implies that $R$ forms a right transversal of $H$ in $G$.
\end{proof}

For a subset $S$ of the group $G$, let $I_2(S)$ be the set of involutions of $S$. Recall the elements $g_1$ and $g_2$ of $R$ defined in~\eqref{Eq12}.

\begin{lemma}\label{LEMH}
We have $|H^{g_1}\cap H|=2$ and $|H^{g_2}\cap H|=8$.
\end{lemma}

\begin{proof}
Recall that
\begin{align*}
a&\colon(x,y)\mapsto(x+1,\,y),\ \ b\colon(x,y)\mapsto(2x,\,y),\ \ c\colon(x,y)\mapsto(x,\,y+1),\ \ d\colon(x,y)\mapsto(x,\,2y),\\
s&\colon(x,y)\mapsto\left(\frac{2x+1}{x+1},\,y\right),\ \ t\colon(x,y)\mapsto\left(\frac{-1}{x},\,y\right),\ \ \alpha\colon (x,y)\mapsto\left(\frac{-x}{x+1},\,\frac{-y}{y+1}\right),\\
\beta&\colon(x,y)\mapsto(y,x),\ \ u=s^\beta,\ \ v=t^\beta,\ \ g_1=a^4c^5,\ \ g_2=a^2c^3d^2.
\end{align*}
It is straightforward to verify that
\begin{align*}
&|a|=|c|=7,\ \ |b|=|d|=3,\ \ s^\alpha=s^3,\ \ t^\alpha=st,\ \ u^\alpha=u^3,\ \ v^\alpha=uv,\ \\
&g_1\colon(x,y)\mapsto(x+4,\,y+5),\ \ g_2\colon(x,y)\mapsto(x+2,\,4y+5).
\end{align*}
According to Lemmas~\ref{LEMorder} and~\ref{LEMmaxsub}, any elements $x$ and $y$ of $H$ can be written as
\begin{equation}\label{Eq20}
x=s^{k_1}t^{\ell_1}u^{m_1}v^{n_1}\alpha^{\epsilon_1}\beta^{\delta_1}\ \text{ and } \ y=s^{k_2}t^{\ell_2} u^{m_2}v^{n_2}\alpha^{\epsilon_2}\beta^{\delta_2}
\end{equation}
for some $k_1,k_2,m_1,m_2\in\{0,1,2,3\}$ and $\ell_1,\ell_2,n_1,n_2,\epsilon_1,\epsilon_2,\delta_1,\delta_2\in\{0,1\}$.
Since $u=s^\beta$, $v=t^\beta$, $s^t=s^{-1}$, $s^\alpha=s^3$, $u^\alpha=u^3$, $t^\alpha=st$, $v^\alpha=uv$, $(st)^{\ell_2}=s^{\ell_2}t^{\ell_2}$ and $(uv)^{n_2}=u^{n_2}v^{n_2}$, we have
\begin{align}\label{Eq21}
\nonumber xy&=s^{k_1}t^{\ell_1}u^{m_1}v^{n_1}\alpha^{\epsilon_1}\beta^{\delta_1}\cdot s^{k_2}t^{\ell_2} u^{m_2}v^{n_2}\alpha^{\epsilon_2}\beta^{\delta_2}\\
\nonumber &=s^{k_1}t^{\ell_1}u^{m_1}v^{n_1}(\alpha^{\epsilon_1}\beta^{\delta_1}s^{k_2}t^{\ell_2} u^{m_2}v^{n_2}\beta^{\delta_1}\alpha^{\epsilon_1})\beta^{\delta_2-\delta_1}\alpha^{\epsilon_2-\epsilon_1}\\
\nonumber &=s^{k_1}t^{\ell_1}u^{m_1}v^{n_1}(s^{3^{\epsilon_1}k_2+\epsilon_1\ell_2}t^{\ell_2}u^{3^{\epsilon_1}m_2+\epsilon_1n_2}v^{n_2})^{\beta^{\delta_1}}
\beta^{\delta_2-\delta_1}\alpha^{\epsilon_2-\epsilon_1},\\
&=\begin{cases}
s^{k_1+(-1)^{\ell_1}(3^{\epsilon_1}k_2+\epsilon_1\ell_2)}t^{\ell_1+\ell_2} u^{m_1+(-1)^{n_1}(3^{\epsilon_1}m_2+\epsilon_1n_2)}v^{n_1+n_2}\beta^{\delta_2}\alpha^{\epsilon_2-\epsilon_1} &\text{ if } \delta_1=0 \\
s^{k_1+(-1)^{\ell_1}(3^{\epsilon_1}m_2+\epsilon_1n_2)}t^{\ell_1+n_2} u^{m_1+(-1)^{n_1}(3^{\epsilon_1}k_2+\epsilon_1\ell_2)}v^{n_1+\ell_2}\beta^{\delta_2-1}\alpha^{\epsilon_2-\epsilon_1} &\text{ if } \delta_1=1.
\end{cases}
\end{align}

First consider elements $x$ of order $2$ in $H$. Since $x^2=1$, taking $x=y$ in~\eqref{Eq21} gives
\[
\begin{cases}
k_1+(-1)^{\ell_1}(3^{\epsilon_1}k_1+\epsilon_1\ell_1)\equiv0\pmod{4}\\
m_1+(-1)^{n_1}(3^{\epsilon_1}m_1+\epsilon_1n_1)\equiv0\pmod{4}\\
\delta_1=0
\end{cases}
\text{or }\
\begin{cases}
k_1+(-1)^{\ell_1}(3^{\epsilon_1}m_1+\epsilon_1n_1)\equiv0\pmod{4}\\
m_1+(-1)^{n_1}(3^{\epsilon_1}k_1+\epsilon_1\ell_1)\equiv0\pmod{4}\\
n_1+\ell_1\equiv0\pmod{2}\\
\delta_1=1.
\end{cases}
\]
Let $N=\langle s,t\rangle\times\langle u,v\rangle$.
It follows that
\begin{align*}
&I_2(\langle s,t\rangle)=\{s^2,t,st,s^2t,s^3t\},\quad I_2(N\alpha)=\{s^ju^k\alpha\mid j,k\in\{0,1,2,3\}\},\\
&I_2(N\beta)=\{s^jtu^jv\beta, s^ju^k\beta\mid j,k\in \{0,1,2,3\},\, j+k\equiv0\pmod{4}\},\\
&I_2(N\alpha\beta)=\{s^ju^j\alpha\beta\mid j\in \{0,1,2,3\}\}\cup\{(stv\alpha\beta)^{\beta^k},(s^2tu^3v\alpha\beta)^{\beta^k}\mid k\in\{0,1\}\}.
\end{align*}
Note that $I_2(\langle u,v\rangle)=I_2(\langle s,t\rangle^{\beta})$. It is straightforward to verify that
\begin{equation}\label{Eq25}
\begin{aligned}
I_2(\langle s,t\rangle^{g_1})={}&\{a^3bs^3t,a^4b^2s^2,b^2t,a^3b^2s,abs^2t\},\\
I_2(\langle u,v\rangle^{g_1})={}&\{cdu,cdu^2v,cd^2u^3v,uv,c^3u^3\},\\
I_2\big((N\alpha)^{g_1}\big)={}&\{a^3bc^4d^2s^3tu\alpha,a^3bc^6ds^3tu^2\alpha,
a^3bcs^3tv\alpha,a^3bc^4s^3tu^3\alpha,a^6bc^4d^2stu\alpha,c^4d^2u\alpha,\\
&a^6c^4d^2s^3u\alpha,a^6bc^6dstu^2\alpha,c^6du^2\alpha,
a^6c^6ds^3u^2\alpha,a^6bcstv\alpha,cv\alpha,a^6cs^3v\alpha,\\
&a^6bc^4stu^3\alpha,c^4u^3\alpha,a^6c^4s^3u^3\alpha\},\\
I_2\big((N\beta)^{g_1}\big)={}&\{a^2bc^3d^2s^2tu^2\beta,a^2b^2c^6d^2s^3tv\beta,ac^2d^2stu\beta,
a^4ds^3u^2v\beta,ac^6\beta,a^4c^5dtuv\beta,\\
&a^2bc^2dsu^3v\beta,a^2b^2c^5s^2u^3\beta\},\\
I_2\big((N\alpha\beta)^{g_1}\big)={}&\{a^5b^2c^2dsu^3v\alpha\beta,bc^5ds^2uv\alpha\beta,
a^2c^6t\alpha\beta,a^5c^5s^3u^3\alpha\beta,a^2c^2d^2u\alpha\beta,a^5c^6d^2s^3tv\alpha\beta,\\
&bds^2tu^2v\alpha\beta,a^5b^2c^3d^2stu^2\alpha\beta\},\\
I_2(\langle s,t\rangle^{g_2})={}&\{a^2bs^3,a^4s,s^2t,b^2t,a^2bst\},\\
I_2(\langle u,v\rangle^{g_2})={}&\{v,u^2v,c^3d^2uv,u^2,c^3d^2u\},\\
I_2\big((N\alpha)^{g_2}\big)={}&\{a^6c^6d^2stuv\alpha,a^6cstv\alpha,
a^6c^6d^2st\alpha,a^6cstu^3\alpha,a^4bc^6d^2uv\alpha,a^4b^2c^6d^2suv\alpha,\\
&a^3c^6d^2s^3uv\alpha,a^4bcv\alpha,a^4b^2csv\alpha,a^3cs^3v\alpha,
    a^4bc^6d^2\alpha,a^4b^2c^6d^2s\alpha,a^3c^6d^2s^3\alpha,\\
&a^4bcu^3\alpha,a^4b^2csu^3\alpha,a^3cs^3u^3\alpha\},\\
I_2\big((N\beta)^{g_2}\big)={}&\{ab^2c^5ds^2tu\beta,bc^3dstu^2v\beta,ab^2c^3s^2v\beta,
bc^4d^2suv\beta,ab^2c^3d\beta,bc^5ds^3u^3v\beta,\\
&ab^2c^4d^2tu^3\beta,bc^3s^3tu^2\beta\},\\
I_2\big((N\alpha\beta)^{g_2}\big)={}&\{a^6bc^6dstuv\alpha\beta,a^3b^2c^5d^2t\alpha\beta,
a^6bc^5u\alpha\beta,a^3b^2cds^3u^3\alpha\beta,a^6bc^6ds^2u^2\alpha\beta,\\
&a^3b^2cdsu^2v\alpha\beta,a^3b^2c^5d^2s^2tu^3v\alpha\beta,a^6bc^5s^3tv\alpha\beta\}.
\end{aligned}
\end{equation}
By Lemma~\ref{LEMmaxsub}, we have $H\cap R=1$ and
\[
I_2(H)=I_2(\langle s,t\rangle)\cup I_2(\langle u,v\rangle)\cup I_2(\langle s,t\rangle)I_2(\langle u,v\rangle)\cup I_2(N\alpha)\cup I_2(N\beta)\cup I_2(N\alpha\beta).
\]
Then we observe from~\eqref{Eq25} that
\[
H\cap I_2(H^{g_1})=\{uv\}\ \text{ and }\ H\cap I_2(H^{g_2})=\{s^2t,v,u^2v,u^2,s^2tv,s^2tu^2v,s^2tu^2\}.
\]
Therefore,
\begin{align}\label{Eq22}
x=
\begin{cases}
uv\ &\text{ if }\ x\in H\cap H^{g_1}\\
s^2t,v,u^2v,u^2,s^2tv,s^2tu^2v\text{ or }s^2tu^2\ &\text{ if }\ x\in H\cap H^{g_2}.
\end{cases}
\end{align}

Next suppose that $x\in H\cap H^{g_j}$ and $|x|=4$ for some $j\in\{1,2\}$. Then $x^2\in H\cap I_2(H^{g_j})$. Let
\begin{align}\label{Eq18}
\chi\colon s^kt^\ell u^mv^n\alpha^\epsilon\beta^\delta\mapsto(-1)^{\ell+n}
\end{align}
be the mapping from $H$ to the group $\{-1,1\}$, where $k,m\in\{0,1,2,3\}$ and $\ell,n,\epsilon,\delta\in\{0,1\}$.
Then for all $y,z\in H$ we derive from~\eqref{Eq21} that
\begin{equation}\label{Eq24}
\chi(yz)=\chi(y)\chi(z),
\end{equation}
that is, $\chi$ is a group homomorphism.
In particular, we have $\chi(z^2)=\chi(z)\chi(z)=1$ for all $z\in H$.
If $x\in H\cap H^{g_1}$, then by~\eqref{Eq22} we see that $x^2=uv$, but by \eqref{Eq24} we obtain
\[
1=\chi(x^2)=\chi(uv)=\chi(u)\chi(v)=1\cdot(-1)=-1,
\]
a contradiction.
Now $x\in H\cap H^{g_2}$, and since $x^2\in H\cap I_2(H^{g_2})$,~\eqref{Eq22} shows that
\[
x^2\in\{s^2t,v,u^2v,u^2,s^2tv,s^2tu^2v,s^2tu^2\}.
\]
Since $\chi(s^2t)=\chi(v)=\chi(u^2v)=\chi(s^2tu^2)=-1$, we have $x^2\notin\{s^2t,v,u^2v,s^2tu^2\}$.
Since $x\in H\cap H^{g_2}$, there exists $y\in H$ such that $x=y^{g_2}$, and so $x^2=(y^2)^{g_2}$.
Note from~\eqref{Eq25} that $u^2=(u^2v)^{g_2}$ and $(s^2tv)=(stu^2)^{g_2}$.
If $x^2\in\{u^2,s^2tv\}$, then $y^2\in \{u^2v,stu^2\}$. However, by~\eqref{Eq24} we have $\chi(u^2v)=\chi(stu^2)=-1$. Hence $x^2\notin\{u^2,s^2tv\}$, and so $x^2=s^2tu^2v$.
This together with~\eqref{Eq20} and~\eqref{Eq21} leads to
\[
\begin{cases}
\ell_1+\ell_1\equiv1\pmod2\\
\delta_1=0
\end{cases}\text{ or }\ \
\begin{cases}
k_1+(-1)^{\ell_1}(3^{\epsilon_1}m_1+\epsilon_1n_1)\equiv2\pmod4\\
m_1+(-1)^{n_1}(3^{\epsilon_1}k_1+\epsilon_1\ell_1)\equiv2\pmod4\\
n_1+\ell_1\equiv1\pmod2\\
\delta_1=1.
\end{cases}
\]
It is easy to see that the former system of equations has no solutions, and the latter has solutions precisely when $(k_1,\ell_1,m_1,n_1,\epsilon_1,\delta_1)$ is one of
\[
(2,0,0,1,0,1),\ (2,1,0,0,0,1),\ (0,1,2,0,0,1)\ \text{ and }\ (0,0,2,1,0,1).
\]
Thus $x\in \{s^2v\beta,s^2t\beta,tu^2\beta,u^2v\beta\}$. For each $y\in H$ such that $y^{g_2}=x$, we have $(y^2)^{g_2}=x^2=s^2tu^2v$, and so $y^2=stv$ by~\eqref{Eq25}. Combining this with~\eqref{Eq21} we derive that
\[
y\in\{st\alpha\beta,v\alpha\beta,s^3tu^2\alpha\beta,s^2u^2v\alpha\beta\}.
\]
However, it is straightforward to verify that
\[
y^{g_2}\in\{a^6bc^6ds^2uv\alpha\beta,a^6bc^6dstu^2\alpha\beta,a^6bc^5s^3tu\alpha\beta,a^6bc^5v\alpha\beta\},
\]
which contradicts the fact $y^{g_2}=x\in\{s^2v\beta,s^2t\beta,tu^2\beta,u^2v\beta\}$.
Therefore, all non-identity elements of $H\cap H^{g_2}$ are involutions.
As a consequence,
\[
|H\cap H^{g_1}|=|\langle uv\rangle|=2,
\]
\[
|H\cap H^{g_2}|=|\langle s^2t\rangle\times\langle u^2,v\rangle|=2\cdot4=8.\qedhere
\]
\end{proof}

\begin{lemma}\label{ISOgraph}
The digraph $\Sigma$ in Construction~$\ref{CONver-pri}$ is isomorphic to $\Cos(G,H,H\{g_1,g_2\}H)$.
\end{lemma}

\begin{proof}
Let $S_1$, $S_2$, $S_3$, $S_4$ and $S$ be as in Construction~$\ref{CONver-pri}$. For each $x\in S$, as listed in Tables~\ref{tab1},~\ref{tab2},~\ref{tab3} and \ref{tab4}, a straightforward calculation verifies that $x=hg_jk$ with $h$, $k$ and $j$ given in the corresponding row.
Therefore, $S$ is a subset of $Hg_1H\cup Hg_2H$, and hence
\begin{equation}\label{Eq6}
\{Hx\mid x \in S\}\subseteq\{Hy\mid y \in Hg_1H\cup Hg_2H\}.
\end{equation}
According to Lemma~\ref{LEMH}, we have
\begin{align*}
|Hg_1H|/|H|&=|H|/|H^{g_1}\cap H|=256/2=128,\\
|Hg_2H|/|H|&=|H|/|H^{g_2}\cap H|=256/8=32.
\end{align*}
Consequently,
\begin{equation}\label{Eq11}
|\{Hy\mid y \in Hg_1H\cup Hg_2H\}|=|Hg_1H|/|H|+|Hg_2H|/|H|=128+32=160.
\end{equation}
Recall from Lemma~\ref{LEMmaxsub} that $R$ forms a right transversal of $H$ in $G$. We then conclude from $S\subseteq R$ and Lemma~\ref{LEMRS} that $|\{Hx\mid x \in S\}|=|S|=160$, which combined with~\eqref{Eq6} and~\eqref{Eq11} yields
\begin{equation}\label{Eq7}
\{Hx\mid x \in S\}=\{Hy\mid y \in Hg_1H\cup Hg_2H\}.
\end{equation}

\begin{table}[htbp]
\begin{multicols}{2}
\caption{$x\in S_1S_3^\beta$ and $j=1$}\label{tab1}
\[
\begin{array}{|l|l|l|}
\hline
x & h & k\\
\hline
ac^3 & su\beta & sv\beta\\
\hline
ad & su^2\beta & s^3tv\beta\\
\hline
acd^2 & sv\beta & s^2v\beta\\
\hline
ac^4d^2 & suv\beta & tv\beta\\
\hline
a^5c^3 & u\beta & s\beta\\
\hline
a^5d & u^2\beta & s^3t\beta\\
\hline
a^5cd^2 & v\beta & s^2\beta\\
\hline
a^5c^4d^2 & uv\beta & t\beta\\
\hline
a^6bc^3 & tu\beta & su^2v\beta \\
\hline
a^6bd & tu^2\beta & s^3tu^2v\beta\\
\hline
a^6bcd^2 & tv\beta & s^2u^2v\beta\\
\hline
a^6bc^4d^2 & tuv\beta & tu^2v\beta\\
\hline
a^6b^2c^3 & s^3u\beta & su^2\beta\\
\hline
a^6b^2d & s^3u^2\beta & s^3tu^2\beta\\
\hline
a^6b^2cd^2 & s^3v\beta & s^2u^2\beta\\
\hline
a^6b^2c^4d^2 & s^3uv\beta & tu^2\beta\\
\hline
\end{array}
\]

\columnbreak

\caption{$x\in S_1(S_3^{-1})^\beta$ and $j=1$}\label{tab2}
\[
\begin{array}{|l|l|l|}
\hline
x & h & k\\
\hline
ac^4 & s\beta & v\beta\\
\hline
ad^2 & su^2v\beta & s^3v\beta\\
\hline
a(cd^2)^{-1} & su^3\beta & stv\beta\\
\hline
a(c^4d^2)^{-1} & su^3v\beta & s^2tv\beta\\
\hline
a^5c^4 & \beta & \beta\\
\hline
a^5d^2 & u^2v\beta & s^3\beta\\
\hline
a^5(cd^2)^{-1} & u^3\beta & st\beta\\
\hline
a^5(c^4d^2)^{-1} & u^3v\beta & s^2t\beta\\
\hline
a^6bc^4 & t\beta & u^2v\beta \\
\hline
a^6bd^2 & tu^2v\beta & s^3u^2v\beta\\
\hline
a^6b(cd^2)^{-1} & tu^3\beta & stu^2v\beta\\
\hline
a^6b(c^4d^2)^{-1} & tu^3v\beta & s^2tu^2v\beta\\
\hline
a^6b^2c^4 & s^3\beta & u^2\beta\\
\hline
a^6b^2d^2 & s^3u^2v\beta & s^3u^2\beta\\
\hline
a^6b^2(cd^2)^{-1} & s^3u^3\beta & stu^2\beta\\
\hline
a^6b^2(c^4d^2)^{-1} & s^3u^3v\beta & s^2tu^2\beta\\
\hline
\end{array}
\]
\end{multicols}
\end{table}

\begin{table}[htbp]
\begin{multicols}{2}
\caption{$x\in S_1S_2^\beta$ and $j=2$}\label{tab3}
\[
\begin{array}{|l|l|l|}
\hline
x & h & k \\
\hline
acd & u^2\alpha & t\alpha \\
\hline
a(cd)^{-1} & u^3\alpha & tu^3\alpha \\
\hline
a^5cd & s^3u^2\alpha & s^3\alpha \\
\hline
a^5(cd)^{-1} & s^3u^3\alpha & s^3u^3\alpha \\
\hline
a^6bcd & su^2\alpha & \alpha \\
\hline
a^6b(cd)^{-1} & su^3\alpha & u^3\alpha \\
\hline
a^6b^2cd & s^2u^2\alpha & s\alpha \\
\hline
a^6b^2(cd)^{-1} & s^2u^3\alpha & su^3\alpha \\
\hline
\end{array}
\]

\columnbreak

\caption{$x\in S_1^{-1}S_4^\beta$ and $j=2$}\label{tab4}
\[
\begin{array}{|l|l|l|}
\hline
x & h & k\\
\hline
a^6c^2d & tu & s^3u\\
\hline
a^6(c^2d)^{-1} & t & s^3\\
\hline
a^2c^2d & u & u\\
\hline
a^2(c^2d)^{-1} & 1 & 1\\
\hline
(a^6b)^{-1}c^2d & su & s^2u\\
\hline
(a^6b)^{-1}(c^2d)^{-1} & s & s^2\\
\hline
(a^6b^2)^{-1}c^2d & s^2u & su\\
\hline
(a^6b^2)^-1(c^2d)^{-1} & s^2 & s\\
\hline
\end{array}
\]
\end{multicols}
\end{table}

Let $\psi\colon r\mapsto Hr$ be the mapping from the vertex set $R$ of $\Sigma$ to $[G{:}H]$.
Next we prove that $\psi$ is a digraph isomorphism from $\Sigma$ to $\Cos(G,H,H\{g_1,g_2\}H)$.
Since $R$ forms a right transversal of $H$ in $G$, we derive that $\psi$ is bijective.
Hence for $r_1$ and $r_2$ in $R$, we have $r_2r_1^{-1}\in S$ if and only if $Hr_2r_1^{-1}\in\{Hx\mid x \in S\}$.
By~\eqref{Eq7}, the latter condition holds if and only if $Hr_2r_1^{-1}\in\{Hy\mid y \in Hg_1H\cup Hg_2H\}$, or equivalently, $r_2r_1^{-1}\in Hg_1H\cup Hg_2H$.
Thus we conclude that $r_1\rightarrow r_2$ is an arc of $\Sigma$ if and only if $Hr_1\rightarrow Hr_2$ is an arc of $\Cos(G,H,H\{g_1,g_2\}H)$.
This shows that $\psi$ is an isomorphism from $\Sigma$ to $\Cos(G,H,H\{g_1,g_2\}H)$.
\end{proof}

Now we give the main result of this section.

\begin{theorem}\label{THEmain4}
For the digraph $\Sigma$ in Construction~$\ref{CONver-pri}$, the following hold:
\begin{enumerate}[\rm(a)]
\item $|V(\Sigma)|=441$;
\item $\Val(\Sigma)=160$;
\item $\Sigma$ is strongly connected;
\item $\Sigma$ is vertex-primitive;
\item $\Sigma$ is non-diagonalizable.
\end{enumerate}
\end{theorem}

\begin{proof}
Since $\Sigma=\Cay(R,S)$, parts~(a) and~(b) follow directly from Lemma~\ref{LEMRS}.
It is straightforward to verify that $a=(a^5cd)^3$, $b=(cb)^7$, $c=(cb)b^{-1}$ and $d=a^{-1}(ad)$.
Since
\[
a^5cd\in S_1S_2^\beta\subseteq S,\ \ cb\in S_1^\beta S_3\subseteq S\ \text{ and }\ ad\in S_1S_3^\beta\subseteq S,
\]
we see that $R=\langle a,b,c,d\rangle\leq\langle S\rangle\leq R$, and so $R=\langle S\rangle$.
This implies that $\Sigma=\Cay(R,S)$ is connected, and so $\Sigma$ is strongly connected (see~\cite[Lemma~2.6.1]{GR2001}). This proves part~(c).
Part~(d) follows from Lemmas~\ref{LEMmaxsub}, ~\ref{ISOgraph} and~\cite[Lemma~2.5.1]{GR2001}.

It remains to prove part~(e). Let $\omega$ be an element of $\mathbb{F}_7^\times$ with order $3$,
let $\zeta\in\mathbb{C}$ be a primitive $7$-th root of unity,
let $V$ be the underlying vector space of the group algebra $\mathbb{C}[\langle\omega\rangle]$,
and let $\varphi(a^kb^\ell)$ be the linear transformation on $V$ such that
\begin{equation}\label{Eq8}
(\omega^j)^{\varphi(a^kb^\ell)}=\zeta^{k2^j}\omega^{j-\ell}\ \text{ for all }j,\ell\in\{0,1,2\}\text{ and }k\in\{0,1,\dots,6\}.
\end{equation}
It follows from $a^{b^{-1}}=a^4$ that
\begin{align*}
(\omega^j)^{\varphi(a^{k_1}b^{\ell_1}a^{k_2}b^{\ell_2})}
&=(\omega^j)^{\varphi(a^{k_1+4^{\ell_1}k_2}b^{\ell_1+\ell_2})}\\
&=\zeta^{(k_1+4^{\ell_1}k_2)2^j}\omega^{j-(\ell_1+\ell_2)}\\
&=\zeta^{k_12^j+2^{j+2\ell_1}k_2}\omega^{j-\ell_1-\ell_2}\\
&=\zeta^{k_12^j+2^{j+2\ell_1-3\ell_1}k_2}\omega^{j-\ell_1-\ell_2}\\
&=\zeta^{k_12^j}\big(\zeta^{2^{j-\ell_1}k_2}\omega^{(j-\ell_1)-\ell_2}\big)\\
&=(\omega^j)^{\varphi(a^{k_1}b^{\ell_1})\varphi(a^{k_2}b^{\ell_2})}.
\end{align*}
Hence $\varphi$ is a representation of $\langle a,b\rangle$ on $V$.
Moreover, since $\beta$ swaps $a$ with $c$ and swaps $b$ with $d$, it follows that $\varphi\circ\beta$ is a representation of $\langle c,d\rangle$ on $V$.
Thus $\rho:=\varphi\otimes(\varphi\circ\beta)$ is a representation of $\langle a,b\rangle\times\langle c,d\rangle=R$ on $V$.
For $X,Y\subseteq\langle a,b\rangle$, we have
\[
\rho(XY^\beta)=\varphi(X)\otimes(\varphi\circ\beta)(Y^\beta)=\varphi(X)\otimes\varphi(Y).
\]
Then Lemma~\ref{LEMRS}(b) implies that
\begin{align*}
\rho(S)&=\rho\big((S_1\sqcup S_1^{-1})(S_3\sqcup S_3^{-1})^\beta\sqcup(S_3\sqcup S_3^{-1})(S_1\sqcup S_1^{-1})^\beta\\
&\phantom{=}\sqcup S_1S_2^\beta\sqcup S_2S_1^\beta \sqcup S_1^{-1}S_4^\beta\sqcup S_4(S_1^{-1})^\beta\big)\\
&=\rho\big((S_1\sqcup S_1^{-1})(S_3\sqcup S_3^{-1})^\beta\big)+\rho\big((S_3\sqcup S_3^{-1})(S_1\sqcup S_1^{-1})^\beta\big)\\
&\phantom{=}+\rho(S_1S_2^\beta)+\rho(S_2S_1^\beta)+\rho(S_1^{-1}S_4^\beta)+\rho\big(S_4(S_1^{-1})^\beta\big)\\
&=\varphi(S_1\sqcup S_1^{-1})\otimes\varphi(S_3\sqcup S_3^{-1})+\varphi(S_3\sqcup S_3^{-1})\otimes\varphi(S_1\sqcup S_1^{-1})\\
&\phantom{=}+\varphi(S_1)\otimes\varphi(S_2)+\varphi(S_2)\otimes\varphi(S_1)+\varphi(S_1^{-1})\otimes\varphi(S_4)+\varphi(S_4)\otimes\varphi(S_1^{-1})\\
&=\big(\varphi(S_1)+\varphi(S_1^{-1})\big)\otimes\big(\varphi(S_3)+\varphi(S_3^{-1})\big)+\big(\varphi(S_3)+\varphi(S_3^{-1})\big)
\otimes\big(\varphi(S_1)+\varphi(S_1^{-1})\big)\\
&\phantom{=}+\varphi(S_1)\otimes\varphi(S_2)+\varphi(S_2)\otimes\varphi(S_1)+\varphi(S_1^{-1})\otimes\varphi(S_4)+\varphi(S_4)\otimes\varphi(S_1^{-1}).
\end{align*}
Thus for all invertible matrices $T,Q\in M_{3\times3}(\mathbb{C})$, we conclude from Lemma~\ref{LEMtenpro} that
\begin{equation}\label{Eq13}
\begin{aligned}
&\phantom{=}(T\otimes Q)^{-1}\big(\rho(S)\big)(T\otimes Q)\\
&=(T^{-1}\otimes Q^{-1})\big(\rho(S)\big)(T\otimes Q)\\
&=\bigg(\Big(T^{-1}\big(\varphi(S_1)+\varphi(S_1^{-1})\big)T\Big)\otimes\Big(Q^{-1}\big(\varphi(S_3)+\varphi(S_3^{-1})\big)Q\Big)\bigg)\\
&\phantom{=}+\bigg(\Big(T^{-1}\big(\varphi(S_3)+\varphi(S_3^{-1})\big)T\Big)\otimes\Big(Q^{-1}\big(\varphi(S_1)+\varphi(S_1^{-1})\big)Q\Big)\bigg)\\
&\phantom{=}+\Big(\big(T^{-1}\varphi(S_1)T\big)\otimes\big(Q^{-1}\varphi(S_2)Q\big)\Big)+
\Big(\big(T^{-1}\varphi(S_2)T\big)\otimes\big(Q^{-1}\varphi(S_1)Q\big)\Big)\\
&\phantom{=}+\Big(\big(T^{-1}\varphi(S_1^{-1})T\big)\otimes\big(Q^{-1}\varphi(S_4)Q\big)\Big)+
\Big(\big(T^{-1}\varphi(S_4)T\big)\otimes\big(Q^{-1}\varphi(S_1^{-1})Q\big)\Big).
\end{aligned}
\end{equation}
Moreover, for all $i\in\{0,1,2\}$, we derive from~\eqref{Eq8} that
\begin{align*}
(\omega^i)^{\varphi(S_1)}
&=(\omega^i)^{\varphi(a)}+(\omega^i)^{\varphi(a^5)}+(\omega^i)^{\varphi(a^6b)}+(\omega^i)^{\varphi(a^6b^2)}\\
&=\zeta^{2^i}\omega^i+\zeta^{5\cdot2^i}\omega^i+\zeta^{6\cdot2^i}\omega^{i-1}+\zeta^{6\cdot2^i}\omega^{i-2},\\
(\omega^i)^{\varphi(S_1^{-1})}
&=(\omega^i)^{\varphi(a^6)}+(\omega^i)^{\varphi(a^2)}+(\omega^i)^{\varphi(a^2b^2)}+(\omega^i)^{\varphi(a^4b)}\\
&=\zeta^{6\cdot2^i}\omega^i+\zeta^{2\cdot2^i}\omega^i+\zeta^{2\cdot2^i}\omega^{i-2}+\zeta^{4\cdot2^i}\omega^{i-1},\\
(\omega^i)^{\varphi(S_2)}
&=(\omega^i)^{\varphi(ab)}+(\omega^i)^{\varphi((ab)^{-1})}=(\omega^i)^{\varphi(ab)}+(\omega^i)^{\varphi(a^5b^2)}
=\zeta^{2^i}\omega^{i-1}+\zeta^{5\cdot2^i}\omega^{i-2},\\
(\omega^i)^{\varphi(S_3)}
&=(\omega^i)^{\varphi(a^3)}+(\omega^i)^{\varphi(b)}+(\omega^i)^{\varphi(ab^2)}+(\omega^i)^{\varphi(a^4b^2)}\\
&=\zeta^{3\cdot2^i}\omega^i+\omega^{i-1}+\zeta^{2^i}\omega^{i-2}+\zeta^{4\cdot2^i}\omega^{i-2},\\
(\omega^i)^{\varphi(S_3^{-1})}
&=(\omega^i)^{\varphi(a^4)}+(\omega^i)^{\varphi(b^2)}+(\omega^i)^{\varphi(a^3b)}+(\omega^i)^{\varphi(a^5b)}\\
&=\zeta^{4\cdot2^i}\omega^i+\omega^{i-2}+(\zeta^{3\cdot2^i}+\zeta^{5\cdot2^i})\omega^{i-1},\\
(\omega^i)^{\varphi(S_4)}
&=(\omega^i)^{\varphi(a^2b)}+(\omega^i)^{\varphi\big((a^2b)^{-1}\big)}=(\omega^i)^{\varphi(a^2b)}+(\omega^i)^{\varphi(a^3b^2)}
=\zeta^{2\cdot2^i}\omega^{i-1}+\zeta^{3\cdot2^i}\omega^{i-2}.
\end{align*}
Hence with respect to the basis $1,\omega,\omega^2$ of $V$ we conclude that
\begin{equation}\label{Eq16}
\begin{aligned}
\varphi(S_1)+\varphi(S_1^{-1})&=
\begin{pmatrix}
\zeta^5+\zeta & \zeta^6 & \zeta^6\\
\zeta^5 & \zeta^2+\zeta^3 & \zeta^5\\
\zeta^3 & \zeta^3 & \zeta^6+\zeta^4\\
\end{pmatrix}
+
\begin{pmatrix}
\zeta^6+\zeta^2 & \zeta^2 & \zeta^4\\
\zeta & \zeta^5+\zeta^4 & \zeta^4\\
\zeta & \zeta^2 & \zeta^3+\zeta
\end{pmatrix},\\
\varphi(S_3)+\varphi(S_3^{-1})&=
\begin{pmatrix}
\zeta^3 & \zeta^4+\zeta & 1\\
1 & \zeta^6 & \zeta^2+\zeta\\
\zeta^4+\zeta^2 & 1 & \zeta^5\\
\end{pmatrix}
+
\begin{pmatrix}
\zeta^4 & 1 & \zeta^5+\zeta^3\\
\zeta^6+\zeta^3 & \zeta & 1\\
1 & \zeta^6+\zeta^5 & \zeta^2\\
\end{pmatrix},\\
\varphi(S_2)&=
\begin{pmatrix}
0 & \zeta^5 & \zeta\\
\zeta^2 & 0 & \zeta^3\\
\zeta^6 & \zeta^4 & 0\\
\end{pmatrix}
\ \text{ and }\ \
\varphi(S_4)=
\begin{pmatrix}
0 & \zeta^3 & \zeta^2\\
\zeta^4 & 0 & \zeta^6\\
\zeta^5 & \zeta & 0\\
\end{pmatrix}.
\end{aligned}
\end{equation}
Let $x_1=\zeta^4+\zeta^3+\zeta+1$, $x_2=\zeta^4-2\zeta^3-2\zeta^2-2\zeta+1$, and $x_3=\zeta^5+2\zeta^4+4\zeta^3+2\zeta^2+\zeta$, and let
\begin{equation}\label{Eq14}
\begin{aligned}
T_1&=\frac{1}{14}
\begin{pmatrix}
1 & -6 & -1\\
\zeta^5+\zeta & -2(2x_3-7\zeta^3) & -(3x_3-7\zeta^3)\\
x_1 & 2(x_1+2\zeta^4+\zeta^2+2) & x_2\\
\end{pmatrix},\\
T_2&=\frac{1}{14}
\begin{pmatrix}
-2 & -2 & 2\\
-2(\zeta^5+\zeta) & -2(3x_3-7\zeta^3) & -x_3\\
-2x_1 & 2x_2 & -\big(2\zeta^4+3(\zeta^3+\zeta^2+\zeta)+2\big)\\
\end{pmatrix}.
\end{aligned}
\end{equation}
It is straightforward to verify that
\begin{equation}\label{Eq15}
\begin{aligned}
T_1^{-1}&=
\begin{pmatrix}
-2(\zeta^4+\zeta^3-2) & 2(\zeta^6-\zeta^5-\zeta^3+\zeta^2) & 2(\zeta^6+\zeta^4-\zeta^2-\zeta)\\
-(\zeta^4+\zeta^3+3) & \zeta^6+\zeta^2 & \zeta^6+\zeta^4\\
4(\zeta^4+\zeta^3+2) & -2(\zeta^6-\zeta^4+\zeta^2-\zeta-1) & -2(\zeta^6-\zeta^5+\zeta^4-\zeta^3-1)\\
\end{pmatrix},\\
T_2^{-1}&=
\begin{pmatrix}
\zeta^4+\zeta^3-2 & -(\zeta^6-\zeta^5-\zeta^3+\zeta^2) & -(\zeta^6+\zeta^4-\zeta^2-\zeta)\\
\zeta^4+\zeta^3+1 & \zeta^4+\zeta+1 & \zeta^5+\zeta^3+1\\
2(\zeta^4+\zeta^3+3) & -2(\zeta^6+\zeta^2) & -2(\zeta^6+\zeta^4)\\
\end{pmatrix}.
\end{aligned}
\end{equation}
Moreover, let
\begin{align*}
y_1&=-93506(\zeta^5+\zeta^2)-152738(\zeta^4+\zeta^3)-147903,\\
y_2&=-9177(\zeta^5+\zeta^2)-13557(\zeta^4+\zeta^3)-58289,\\
y_3&=56798(\zeta^5+\zeta^2)+98510(\zeta^4+\zeta^3)-85253,\\
y_4&=75152(\zeta^5+\zeta^2)+125624(\zeta^4+\zeta^3)+31325,
\end{align*}
\[
T_3=
\begin{pmatrix}
0 & 4 & 0 & 0 & 4 & -1\\
0 & 4 & 2 & 0 & 0 & 1\\
0 & 455836 & y_1 & 0 & 2y_1+455836 & y_1+113959 \\
0 & 8y_2 & -y_1 & 0 & 2y_3 & y_4\\
1 & 0 & 0 & 0 & 0 & 0\\
1 & 0 & 0 & 1 & 0 & 0\\
\end{pmatrix}.
\]
By a straightforward calculation, we deduce from~\eqref{Eq13}--\eqref{Eq15} that
\begin{align*}
&\phantom{=}\begin{pmatrix}
I_3 & 0\\
0 & T_3\\
\end{pmatrix}
(T_1\otimes T_2)^{-1}\big(\rho(S)\big)(T_1\otimes T_2)\\
&=
\begin{pmatrix}
I_3 & 0\\
0 & T_3\\
\end{pmatrix}
\Bigg(\bigg(\Big(T_1^{-1}\big(\varphi(S_1)+\varphi(S_1^{-1})\big)T_1\Big)\otimes\Big(T_2^{-1}\big(\varphi(S_3)+\varphi(S_3^{-1})\big)T_2\Big)\bigg)\\
&\phantom{=}+\bigg(\Big(T_1^{-1}\big(\varphi(S_3)+\varphi(S_3^{-1})\big)T_1\Big)\otimes\Big(T_2^{-1}\big(\varphi(S_1)
+\varphi(S_1^{-1})\big)T_2\Big)\bigg)\\
&\phantom{=}+\Big(\big(T_1^{-1}\varphi(S_1)T_1\big)\otimes\big(T_2^{-1}\varphi(S_2)T_2\big)\Big)+
\Big(\big(T_1^{-1}\varphi(S_2)T_1\big)\otimes\big(T_2^{-1}\varphi(S_1)T_2\big)\Big)\\
&\phantom{=}+\Big(\big(T_1^{-1}\varphi(S_1^{-1})T_1\big)\otimes\big(T_2^{-1}\varphi(S_4)T_2\big)\Big)+
\Big(\big(T_1^{-1}\varphi(S_4)T_1\big)\otimes\big(T_2^{-1}\varphi(S_1^{-1})T_2\big)\Big)\Bigg)\\
&=\frac{1}{2}
\begin{pmatrix}
A & 0 & 0 &0\\
0 & B & 0 &0\\
0 & 0 & C &0\\
0 & 0 & 0 &D\\
\end{pmatrix}
\begin{pmatrix}
I_3 & 0\\
0 & T_3\\
\end{pmatrix},
\end{align*}
where
$
A=
\begin{pmatrix}
-16 & 0 & 0\\
0 & 10 & 5\\
0 & 40 & 10\\
\end{pmatrix},\
B=
\begin{pmatrix}
8 & 0\\
0 & 8\\
\end{pmatrix},\
C=
\begin{pmatrix}
-16 & 24\\
0 & -16\\
\end{pmatrix}
\ \text{ and }\
D=
\begin{pmatrix}
0 & -10\\
-10 & 20\\
\end{pmatrix}.
$
Since $C$ is non-diagonalizable, we conclude that $(T_1\otimes T_2)^{-1}\big(\rho(S)\big)(T_1\otimes T_2)$ is non-diagonalizable as $T_3$ is invertible, and hence $\rho(S)$ is non-diagonalizable. By Lemma~\ref{LEMrepre2}, this implies that $\Sigma$ is non-diagonalizable, which completes the proof of part~(e).
\end{proof}

\section{Concluding remarks}\label{sec5}

Let $\Gamma_s$ and $\Sigma$ be as in Constructions~\ref{CONs-arc} and~\ref{CONver-pri}, respectively. According to Theorem~\ref{THEmain3}, the digraph $\Gamma_s$ is non-diagonalizable and $s$-arc-transitive for each positive integer $s\geq2$. Moreover, Theorem~\ref{THEmain4} asserts that $\Sigma$ is non-diagonalizable and vertex-primitive.
Combining these with Lemma~\ref{LEMTensorofdi}, we obtain Theorem~\ref{THEinfinite} immediately.

Besides the properties listed in Theorem~\ref{THEinfinite}, we remark that $\Sigma^{\times n}$ is connected since it is vertex-primitive (for otherwise its connected components would form an invariant partition). Moreover, since the out-valency of $\Gamma_s$ is $2$, it follows that the out-valency of $\Gamma_s^{\times n}$ is $2^n$. Hence $\Gamma_s^{\times n}$ is not a disjoint union of digraphs isomorphic to $\Gamma_s^{\times m}$ for any $m<n$. This means that $\Gamma_s^{\times n}$ with $n\geq 1$ is a genuine infinite family of digraphs.

We also remark that the digraph $\Sigma$ in Construction~\ref{CONver-pri} was first discovered by computer search in {\sc Magma}~\cite{Magma}. Although the proof of all the properties of $\Sigma$ in this paper is computer-free, the arguments therein (mostly calculations) have been confirmed by computation in {\sc Magma}~\cite{Magma}.
Further computation in {\sc Magma}~\cite{Magma} shows that $\Gamma_2$ has the smallest order among non-diagonalizable $2$-arc-transitive digraphs (note that $\Gamma_2$ has order $16$), while the smallest order among non-diagonalizable $3$-arc-transitive digraphs is $20$ (note that $\Gamma_3$ has order $48$). Thus a natural question to ask is as follows.

\begin{question}
For $s\geq4$, what is the smallest order of a non-diagonalizable $s$-arc-transitive digraph?
\end{question}

In a similar fashion, one may ask:

\begin{question}
What is the smallest order of a non-diagonalizable vertex-primitive digraph?
\end{question}
Recall that the digraph $\Sigma$ has $441$ vertices (see Lemma~\ref{LEMRS}(a)).
By a non-exhaustive search in {\sc Magma}~\cite{Magma} for non-diagonalizable vertex-primitive digraphs $\Gamma$ of order smaller than $441$, we obtain the following examples $\Gamma=\Cos(G,H,D)$:
\begin{enumerate}[\rm (a)]
\item $|V(\Gamma)|=153$, $G\cong\PSL(2,17)$, $H\cong D_{16}$ and $D=H\{g_1,g_2\}H$ with $g_1,g_2\in G$, where exactly one of $Hg_1H$ and $Hg_2H$ is inverse-closed;
\vskip0.08in
\item $|V(\Gamma)|=165$, $G\cong M_{11}$, $H\cong\GL(2,3)$ and $D=H\{g_1,g_2,g_3\}H$ with $g_1,g_2,g_3\in G$, where exactly one of $Hg_1H$, $Hg_2H$ and $Hg_3H$ is inverse-closed;
\vskip0.08in
\item $|V(\Gamma)|=234$, $G\cong\PSL(3,3)$, $H\cong \Sym(4)$ and $D=H\{g_1,g_2\}H$ with $g_1,g_2\in G$, where neither $Hg_1H$ nor $Hg_2H$ is inverse-closed;
\vskip0.08in
\item $|V(\Gamma)|=325$, $G\cong\PSL(2,25)$, $H\cong D_{24}$ and $D=H\{g_1,g_2\}H$ with $g_1,g_2\in G$, where exactly one of $Hg_1H$ and $Hg_2H$ is inverse-closed.
\end{enumerate}

It is worth remarking that none of the digraphs in~(a)--(d) is Cayley or arc-transitive, and we do not know any computer-free proof of the non-diagonalizability of them.
Moreover, computation in {\sc Magma}~\cite{Magma} shows that there is no non-diagonalizable vertex-primitive arc-transitive digraph with no more than 1000 vertices.
In light of this, we would like to propose the following conjecture.
\begin{conjecture}
Every vertex-primitive arc-transitive digraph is diagonalizable.
\end{conjecture}

\noindent\textsc{Acknowledgement.}
The authors are grateful to the anonymous referees and Shasha Zheng for their valuable comments and suggestions.
The work was done during a visit of the fourth author to The University of Melbourne. The fourth author would like to thank The University of Melbourne for its hospitality and Beijing Normal University for consistent support.
The first author was supported by the Melbourne Research Scholarship provided by The University of Melbourne.
The fourth author was supported by the scholarship No.~202106040068 from the China Scholarship Council.

\end{document}